\documentclass[12pt]{amsart}
\usepackage{amsmath,amssymb,amsfonts,amsthm,ascmac,amscd}
\usepackage{graphicx}
\usepackage{comment}
\usepackage{xspace}
\usepackage{physics}
\usepackage{indentfirst}
\usepackage{mathtools}
\usepackage{enumerate}
\usepackage{overpic} 

\usepackage{hyperref}
\usepackage{xcolor}
\hypersetup{
	bookmarksnumbered=true, 
	colorlinks=false, 
	citecolor=blue,
	linkcolor=red,
	urlcolor=orange,
}
\numberwithin{equation}{section}
\usepackage{cleveref}
\crefname{def}{Definition}{Definitions}
\crefname{theorem}{Theorem}{Theorems}
\crefname{lemma}{Lemma}{Lemmas}
\crefname{corollary}{Corollary}{Corollaries}
\crefname{proposition}{Proposition}{Propositions}
\crefname{example}{Example}{Examples}
\crefname{remark}{Remark}{Remarks}
\crefname{question}{Question}{Questions}
\crefname{equation}{}{}
\crefname{figure}{Figure}{Figures}

\theoremstyle{definition}
\newtheorem{theorem}{Theorem}[section]
\newtheorem*{theorem*}{Theorem}
\newtheorem{definition}[theorem]{Definition}
\newtheorem*{definition*}{Definition}
\newtheorem{prop}[theorem]{Proposition}
\newtheorem*{prop*}{Proposition}

\newtheorem*{example*}{Example}
\newtheorem{remark}[theorem]{Remark}
\newtheorem*{remark*}{Remark}
\newtheorem{lemma}[theorem]{Lemma}
\newtheorem*{lemma*}{Lemma}
\newtheorem{corollary}[theorem]{Corollary}
\newtheorem*{corollary*}{Corollary}

\newtheorem*{question*}{Question}

\newtheorem*{conjecture*}{Conjecture}

\newtheorem*{exercise*}{Exercise}

\newtheorem*{claim*}{Claim}

\newtheorem*{theorema*}{Main Theorem}

\newtheorem*{propa*}{Proposition}

\newcommand{\N}{\mathbb{N}}
\newcommand{\R}{\mathbb{R}}

\DeclareMathOperator{\id}{id}

\newcommand{\Teich}{\mathcal{T}}

\DeclareMathOperator{\QD}{\mathrm{QD}}
\DeclareMathOperator{\SQD}{\mathrm{SQD}}
\DeclareMathOperator{\BQD}{\mathrm{BQD}}
\DeclareMathOperator{\MF}{\mathcal{MF}}
\DeclareMathOperator{\PMF}{\mathcal{PMF}}
\DeclareMathOperator{\inj}{\mathrm{inj}}

\newcommand{\bz}{\bar{z}}
\newcommand{\Preg}{P_\mathrm{reg}}
\newcommand{\Ppole}{P_\mathrm{pole}}

\allowdisplaybreaks[4]
\title[Harmonic maps compactification and punctured surfaces]{The harmonic maps compactification of Teichm\"{u}ller spaces for punctured Riemann surfaces}
\author{Kento Sakai}
\address{Graduate~School~of~Science, Osaka~University}
\email{u741819k@ecs.osaka-u.ac.jp}
\date{\today}
\begin{document}
\begin{abstract}
  Wolf gave a homeomorphism from the Teichmüller space to the space of quadratic differentials on a closed Riemann surface by using harmonic maps. Moreover, using harmonic maps rays, he gave a compactification of the Teichmüller space and show that it coincides with the Thurston compactification. In this paper, we extend the harmonic maps compactification to the Teichmüller spaces of punctured Riemann surfaces, and show that it still coincides with the Thurston compactification.
\end{abstract}

\maketitle

\section{Introduction}

Let $S$ be a finite type surface of genus $g$ with $n$ punctures such that $\chi(S)=2-2g-n<0$. The \textit{Teichmüller space} of $S$ is defined as the set of complete finite-area hyperbolic metrics on $S$ modulo isotopy. Let $\Teich_{g,n}$ denote the Teichmüller space of $S$.

For $\sigma, \rho\in\Teich_{g,n}$, there exists a unique harmonic diffeomorphism $h(\sigma,\rho)$ from  $(S,\sigma)$ to $(S,\rho)$ isotopic to the identity with finite energy \cite{lohkamp1991harmonic}.
Then the $(2,0)$-part of the pull-buck metric $h(\sigma,\rho)^\ast\rho$ is a holomorphic quadratic differential on $(S,\sigma)$, which is called the \textit{Hopf differential} of $h(\sigma,\rho)$.
Therefore fixing the domain metric $\sigma$, we obtain a map $\Phi\colon\Teich_{g,n}\to\QD(\sigma)$ given by $\Phi(\rho)=(h(\sigma,\rho)^\ast\rho)^{2,0}$, where $\QD(\sigma)$ denotes the vector space of holomorphic quadratic differentials with finite $L^1$-norm on $(S,\sigma)$. It is shown that $\Phi$ is a homeomorphism by Wolf for $n=0$, by Lohkamp for $n>0$ \cite{wolf1989teichmuller,lohkamp1991harmonic}.

Using the map $\Phi$, we can obtain the compactification of $\Teich_{g,n}$ since the vector space $\QD(\sigma)$ has the natural compactification which is obtained by adding an infinity point to each ray from the origin.
This compactification is called the \textit{harmonic maps compactification} and denote by $\overline{\Teich_{g,n}^{\textrm{harm}}}$. The boundary of the harmonic maps compactification is identified with the unit sphere $\SQD(\sigma)$ in $\QD(\sigma)$. On the other hand, Thurston introduced the Thurston compactification $\overline{\Teich_{g,n}^{\mathrm{Th}}}$, which is given by embedding $\Teich_{g,n}$ to the projective space of functionals on the set $\mathcal{C}=\mathcal{C}(S)$ of isotopy classes of essential simple closed curves on $S$.
The boundary $\partial\overline{\Teich_{g,n}^{\mathrm{Th}}}$ of the Thurston compactification is the projective space $\PMF$ of measured foliations on $S$.
The action of the mapping class group on $\Teich_{g,n}$ extends contiunously to $\partial\overline{\Teich_{g,n}^{\mathrm{Th}}}$.

Wolf showed that, if $S$ is closed, the homeomorphism $\Phi$ continuously extends to the homeomorphism $\overline{\Phi}$ from $\overline{\Teich_{g,0}^{\mathrm{Th}}}$ to $\overline{\Teich_{g,0}^{\mathrm{harm}}}$.
This implies that the restriction of $\overline{\Phi}$ to the Thurston boundary coincides with the canonical identification $\overline{F_v}$ of $\SQD(\sigma)$ and $\PMF$, where $\overline{F_v}$ is induced by the map $F_v\colon\QD(\sigma)\to\MF$ given by the vertical measured foliation $F_v(\Psi)$ for $\Psi\in\QD(\sigma)$.
Therefore, the harmonic maps compactification is independent on the choice of the domain metric $\sigma$. In addition, the extension $\overline{\Phi}$ provides the Thurston compactification with a global parametrization by the closed unit ball in $\QD(\sigma)$.

In this paper, we extend the compatibility of the harmonic maps compactification and the Thurston compactification for the Teichmüller spaces of punctured surfaces.

\begin{theorema*}
	The homeomorphism $\Phi\colon \Teich_{g,n}\to\QD(\sigma)$ continuously extends to a homeomorphism from $\overline{\Teich_{g,n}^{\mathrm{Th}}}$ to $\overline{\Teich_{g,n}^{\mathrm{harm}}}$. Moreover, this extension on the boundary coincides with the inverse of the canonical identification $\overline{F_v}$.
\end{theorema*}

Our strategy is based on Wolf's proof for a closed Riemann surface. However, the compactness of the surface is essentially used in his proof. In our setting, a holomorphic quadratic differential may have a simple pole at a puncture, so we have to substaintially change his proof.

We outline our proof below, comparing with Wolf's proof. Fixing $\Phi_0\in\SQD(\sigma)$, we set $\rho_t\coloneqq\Phi^{-1}(t\Phi_0)$ for $t>0$. The one-parameter family $\{\rho_t\}$ is called the \textit{harmonic maps ray} in the direction of $\Phi_0$.
We first show that, for the harmonic diffeomorphsim $h(t)\colon(S,\sigma)\to(S,\rho_t)$, the norm of the Beltrami differential of $h(t)$ converges monotonically to $1$ as $t\to\infty$
(\Cref{prop:NormConvergence}).
For the proof of this convergence, \cite{wolf1989teichmuller} uses the compactness of the closed surface.
Therefore, using some results on holomorphic energy functions for punctured surfaces in \cite{lohkamp1991harmonic}, we prove \Cref{prop:NormConvergence} without the compactness.
From \Cref{prop:NormConvergence}, we obtain the asymptotic length of arcs along the leaves of the horizontal or vertical measured foliation of $\Phi_0$ (\Cref{prop:LengthConvergence}).

Let $\mathcal{C}=\mathcal{C}(S)$ be the set of isotopy classes of essential simple closed curves on $S$ and $\beta$ be the identification of $\QD(\sigma)$ and $\MF$ given by $\beta\Phi=F_v(4\Phi)$.
We next recall the fundamental lemma, which is an inequality of the intersection number function $i(\beta\Phi(\rho),\cdot)$ and the hyperbolic length function $\ell_\rho$ on $\mathcal{C}$ for a metric $\rho\in\Teich_{g,0}$.
\begin{theorem*}[{\cite[Lemma 4.1]{wolf1989teichmuller}}]
	Fix an isotopy class $[\gamma]\in\mathcal{C}$. Then, for every $\rho\in\Teich_{g,0}$, there exist positive constants $k_0=k_0(\|\Phi(\rho)\|)$ and $\eta(\|\Phi(\rho)\|,[\gamma])$ such that
	\begin{equation*}
		i(\beta\Phi(\rho),[\gamma])\leq\ell_{\rho}([\gamma])\leq k_0i(\beta\Phi(\rho),[\gamma])+\eta
	\end{equation*}
	where $k_0\searrow 1$ and $\eta\|\Phi(\rho)\|^{-1/2}\to 0$ as $\|\Phi(\rho)\|\to\infty$.
\end{theorem*}
The key to show the lemma is constructing a ``staircase" representative $\gamma_\Phi$ for $\Phi\in\SQD(\sigma)$. The representative $\gamma_\Phi$ satisfies some conditions. For example, it consists of curves along leaves of horizontal or vertical foliation of $\Phi$ and the intersection number $i(\beta\Phi(\rho),\gamma_\Phi)$ is equal to $i(\beta\Phi(\rho),[\gamma])$. However, for punctured surfaces, we cannot construct $\gamma_\Phi$ in the same manner as \cite{wolf1989teichmuller}, since a quadratic differential may have a simple pole at each puncture. Hence, we show the following instead for punctured surfaces.
\setcounter{section}{4}
\begin{theorem}
	Fix an isotopy class $[\gamma]\in\mathcal{C}$ and $\varepsilon>0$. Then, there exists a nonnegative number $c_0=c_0([\gamma],\varepsilon)<\varepsilon$ satisfying the following: for every $\rho\in\Teich_{g,n}$, there exist positive constants $k_0=k_0(\|\Phi(\rho)\|,\varepsilon)$ and $\eta=\eta(\|\Phi(\rho)\|,[\gamma],\varepsilon)$ such that
	\begin{equation*}
		i(\beta\Phi(\rho),[\gamma])\leq\ell_\rho([\gamma])\leq k_0i(\beta\Phi(\rho),[\gamma])+\eta
	\end{equation*}
	where $k_0\searrow 1$ and $\eta\|\Phi(\rho)\|^{-1/2}\to c_0$ as $\|\Phi(\rho)\|\to\infty$.
\end{theorem}
\setcounter{section}{1}
\noindent
This propositon is different from Wolf's lemma in that the limit of $\eta\|\Phi(\rho)\|^{-1/2}$ as $\|\Phi(\rho)\|\to\infty$ may not be $0$. Setting the allowance for arbitraly small $\varepsilon>0$, we can construct such a ``staircase" representative $\gamma_\Phi$. However, unlike the above lemma, we need to increase the horizontal measure of the representative $\gamma_\Phi$ from $i(\beta\Phi(\rho),[\gamma])$ by a little. Therefore, in order to estimate the length of the additional horizontal arcs, we use some results on quadratic differential metrics in \cite{minsky1992harmonic,wolf1989teichmuller}.
\Cref{prop:FundamentalInequality} is a week version of \cite[Lemma 4.1]{wolf1989teichmuller}, but still leads to the main theorem.

\subsection*{Acknowlagement} The author would like to thank my supervisor, Shinpei Baba, for many discussions and helpful advice.
\section{Background}

\subsection{Harmonic maps}

Let $\overline{S}$ be a closed, oriented, connected surface of genus $g$, and $S\coloneqq\overline{S}-\{p_1,\ldots,p_n\}$.
If $\chi(S)=2-2g-n<0$, the surface admits (complete, finite area) hyperbolic metrics.
Throughout this paper, we fix $S$ such that $\chi(S)<0$ and $S$ is not a thrice punctured surface.
Let $\sigma$ be a hyperbolic metric on $S$.
The hyperbolic surface $(S,\sigma)$ can be regarded as a (punctured) Riemann surface by taking the isothermal coordinates system for $\sigma$.

Let $\sigma|dz|^2, \rho|dw|^2$ be hyperbolic metrics on $S$, where $z=x+iy$ and $w=u+iv$ denote the conformal structures of $(S,\sigma)$ and $(S,\rho)$, respectively. For a $C^2$ map $f\colon(S,\sigma|dz|^2)\to(S,\rho|dw|^2)$, we define the \textit{energy density} of $f$ by
\begin{equation*}
  e(f)\coloneqq \frac{\rho(f(z))}{\sigma(z)}(|f_z|^2+|f_{\bz}|^2),
\end{equation*}
the \textit{total energy} of $f$ by
\begin{equation*}
  \mathcal{E}(f)\coloneqq\int_{S} \rho(f(z))(|f_z|^2+|f_{\bz}|^2)\, \frac{i}{2}dzd\bz
\end{equation*}
and, the \textit{holomorphic energy} $H(f)$ and  \textit{anti-holomorphic energy} $L(f)$ by
\begin{align*}
  H(f) \coloneqq \frac{\rho(f(z))}{\sigma(z)}|f_z|^2, \
  L(f) \coloneqq\frac{\rho(f(z))}{\sigma(z)}|f_{\bz}|^2.
\end{align*}
Let $J(f)$ be the Jacobian of $f$, then $J(f)=H(f)-L(f)$. The Beltrami differential of $f$ is defined by
\begin{equation*}
  \nu(f)\coloneqq \frac{f_{\bz} {d\bz}}{f_zdz}.
\end{equation*}
 The norm of the Beltrami differential of $f$ is the well-defined function on $(S,\sigma)$. Clearly, we have $|\nu(f)|^2=L(f)/H(f)$.

A $C^2$ map $f\colon (S,\sigma)\to(S,\rho)$ is said to be \textit{harmonic} if $f$ is a critical point of the energy functional $\mathcal{E}$. By considering the variation problem for energy, $f$ is a harmonic map if and only if $f$ satisfies the Euler-Lagrange equation
\begin{equation*}
  f_{z\bz}+\frac{\rho_w}{\rho}f_zf_{\bz}=0 .
\end{equation*}

It is a well-known fact that if $S$ is closed, there exists a unique harmonic diffeomorphism to the identity map from $(S,\sigma)$ to $(S,\rho)$. (Eells and Sampson  proved the existence of a harmonic map in each homotopy class \cite{eells1964harmonic}. Hartman proved its uniquness \cite{hartman1967homotopic}. Shoen-Yau and Sampson indepndently proved that the harmonic map homotopic to the identity is a diffeomorphism \cite{schoen1978univalent,sampson1978some}.)

In the case of punctured surfaces, Lohkamp showed the following.

\begin{theorem}[\cite{lohkamp1991harmonic}]
  Let $S$ be a surface of finite type, and $\sigma,\rho$ be complete hyperbolic metrics of finite area on $S$.
  Then, there exists a unique harmonic diffeomorphism $h(\sigma,\rho)\colon (S,\sigma)\to(S,\rho)$ such that $h(\sigma,\rho)$ is homotopic to the identity and $\mathcal{E}(h(\sigma,\rho))<+\infty$.
\end{theorem}

\subsection{Compactifications of the Teichm\"{u}ller space}\label{subsection:CompactificationsOfTeichmüllerSpace}

The \textit{Teichmüller space} of the surface $S$ is the quotient space of hyperbolic metrics on $S$ by the pull-buck action of diffeomorphisms isotopic to the identity. Let $\Teich_{g,n}$ denote the Teichmüller space of $S$. We often write $\rho$ simply for $[\rho]\in\Teich_{g,n}$.

A holomorphic quadratic differential $\Phi$ on $(S,\sigma)$ is said to be \textit{integrable}, if the $|\Phi|$-area of $(S,\sigma)$ is finite, namely
\begin{equation*}
  \|\Phi\|= \int_{S} |\Phi| dxdy < +\infty.
\end{equation*}
Let $\QD(\sigma)$ denote the space of integrable holomorphic quadratic differentials on $(S,\sigma)$. For a holomorphic quadratic differential $\Phi$ on $(S,\sigma)$, $\Phi$ is integrable if and only if $\Phi$ has poles of at most one degree at punctures.

Fixing the metric $\sigma$ of the domain surface, we set $h(\rho)\coloneqq h(\sigma,\rho)$ for $\rho$. Since $(h(\rho)^\ast\rho)^{2,0}$ is a holomorphic quadratic differenttial on $(S,\sigma)$ with finite norm   (\cite{lohkamp1991harmonic} Lemma 7), we can define the map $\Phi$ as
\begin{equation*}
  \Phi : \Teich_{g,n}\to \QD(\sigma) ;  [\rho]\mapsto(h(\rho)^\ast\rho)^{2,0}.
\end{equation*}
The uniqueness of the harmonic diffeomorphism shows that the map $\Phi$ is well-defined. Here, we list very useful well-known formulae.

\begin{prop}\label{prop:formulae}
  Let $J(\rho)=J(h(\rho)), H(\rho)=H(h(\rho)), L(\rho)=L(h(\rho))$, and $\nu(\rho)=\nu(h(\rho))$. Then the following hold:
  \begin{enumerate}[(I)]
    \item $J(\rho)=H(\rho)-L(\rho)$
    \item $|\Phi(\rho)|/\sigma^2=H(\rho)L(\rho)$
    \item $|\nu(\rho)|H(\rho)\sigma=|\Phi(\rho)|$
    \item $\nu(\rho)=L(\rho)/H(\rho)$
    \item $\Delta_\sigma H(\rho)=2H(\rho)-2L(\rho)-2$
    \item $\Delta_\sigma L(\rho)=2L(\rho)-2H(\rho)-2$ on $S-\{\Phi(\rho)= 0\}$
  \end{enumerate}
  Here,
  \begin{equation*}
    \Delta_\sigma\coloneqq\frac{4}{\sigma}\frac{\partial^2}{\partial z\partial\bz}
  \end{equation*}
  is the \textit{Laplace-Beltrami operator} on $(S,\sigma)$.
\end{prop}
\noindent For the proof of (V) and (VI), see \cite[Lemma 3.10.1]{jost2013compact}.

The following theorem is shown by Wolf for $n=0$ and by Lohkamp for $n>0$.
\begin{theorem}[\cite{wolf1989teichmuller}, \cite{lohkamp1991harmonic}]
  $\Phi$ is a homeomorphism.
\end{theorem}

Using the Riemann-Roch theorem, we see that $\QD(\sigma)$ is a vector space of real dimension $6g-6+2n$. Since a vector space can be compactified by adding an infinite point to the endpoint of every ray from the origin, we can obtain the compactification of $\Teich_{g,n}$ through the homeomorphism $\Phi$. The compactification of $\Teich_{g,n}$ is called the \textit{harmonic maps compactification} and denoted by $\overline{\Teich^{\mathrm{harm}}_{g,n}}$. In other words,
\begin{equation*}
  \overline{\Teich^{\mathrm{harm}}_{g,n}} = \BQD(\sigma) \cup \SQD(\sigma),
\end{equation*}
where $\BQD(\sigma)\coloneqq \{\Phi \in \QD(\sigma) \mid \|\Phi\| < 1\}$ and $\SQD(\sigma)\coloneqq \{\Phi \in \QD(\sigma)\mid \|\Phi\|=1 \}$.

On the other hand, Thurston introduced a compactification of $\Teich_{g,n}$, which is called the \textit{Thruston compactification} and denoted by $\overline{\Teich^{\mathrm{Th}}_{g,n}}$.
Here, we put a brief description of $\overline{\Teich_{g,n}^{\mathrm{Th}}}$ with reference to \cite{fathi2012thurston}.

A simple closed curve $\gamma$ on $S$ is \textit{peripheral} if $\gamma$ bounds a punctured disk.
If a simple closed curve $\gamma$ is neither null-homotopic nor peripheral, $\gamma$ is said to be \textit{essential}.
Let $\mathcal{C}=\mathcal{C}(S)$ be the set of  homotopy classes of essential simple closed curves in $S$.
Given functionals $f,g \in \R_{\geq0}^{\mathcal{C}}$, $f$ and $g$ are  said to be equivalent, if there exists a positive real number $\lambda>0$ such that $f=\lambda g$.
The quotient space of the functionals by the equivalence relation is denoted by $P(\R_{\geq0}^{\mathcal{C}})$, and let $\pi:\R_{\geq0}^{\mathcal{C}}-\{0\}\to P(\R_{\geq0}^{\mathcal{C}})$ be the projection.
Let $\MF^\ast$ be the set of nontrivial measured foliations on $S$ which may have possibly one pronged singularities at the punctures. Then, $\MF^\ast$ can be embedded into $\R_{\geq0}^{\mathcal{C}}-\{0\}$ by the map $I_\ast$ which is defined by
\begin{equation*}
  I_\ast(F)\coloneqq (i(F,\cdot) : \mathcal{C} \to \R_{\geq 0})\qquad(F \in \MF^\ast),
\end{equation*}
where $i(F,[\gamma])$ is the infimum of the transeverse measure of representatives of $[\gamma]\in\mathcal{C}$. Thus we identify $\MF^\ast$ with its image in $\R_{\geq 0}^{\mathcal{C}}$. Thurston showed that $\pi\circ I_\ast (\MF^\ast)$ is homeomorphic to a sphere of dimension $6g-7+2n$.

For $\rho \in \Teich_{g,n}$, the length functional $\ell_\ast(\rho)\in \R_{\geq0}^{\mathcal{C}}$ is given by
\begin{equation*}
  \ell_\ast(\rho)([\gamma])\coloneqq\ell_\rho([\gamma])=\inf_{\gamma\in[\gamma]}\ell_{\rho}(\gamma)\qquad([\gamma]\in\mathcal{C}).
\end{equation*}
It is known that $\ell_\ast$ is an embedding of $\Teich_{g,n}$ into $\R_{\geq 0}^{\mathcal{C}} - \{0\}$ and  $\pi\circ\ell_\ast$ is still an embedding of $\Teich_{g,n}$ into $P(\R_{\geq 0}^{\mathcal{C}})$.

Let $\overline{\Teich_{g,n}^{\mathrm{Th}}}$ be the subset $\pi\circ\ell_{\ast}(\Teich_{g,n})\cup \pi\circ I_\ast(\MF^\ast)$ in $P(\R_{\geq0}^{\mathcal{C}})$. We call $\overline{\Teich_{g,n}^{\mathrm{Th}}}$ the \textit{Thuston compactification} of $\Teich_{g,n}$.
In fact, it is a manifold with boundary and homeomorphic to a closed ball of dimension $6g-6+2n$. The boundary $\pi\circ I_\ast(\MF^\ast)$ of the Thurston compactification is denoted by $\PMF$ and called the \textit{Thrston boundary}. By the construction, the mapping class group action on $\Teich_{g,n}$ extends continuously to the Thurston compactification.

\section{Deformation along harmonic maps rays}

\subsection{Norm functions of Beltrami differentials}

We denote the inverse homeomorphism $\Phi^{-1}:\QD(\sigma)\to \Teich_{g,n}$ by $\rho$ .

\begin{definition}
  Let $\Phi_0\in \SQD(\sigma)$. The \textit{harmonic maps ray} in the direction of $\Phi_0$ is the ray defined by $\{\rho_t\coloneqq \rho(t\Phi_0)\}_{t>0}$ in $\Teich_{g,n}$.
\end{definition}
Let $h(t)$ denote the unique harmonic diffeomorphism $h(\rho_t)$ homotopic to the identity with $\mathcal{E}(h(\rho_t))<+\infty$. We denote  the holomorphic energy by $H(t)$, anti-holomorphic energy by $L(t)$, and the Beltrami differentials of $h(t)$ by $\nu(t)$. Clearly $\Phi(\rho_t)=t\Phi_0$ by the definition. The main purpose of this subsection is to prove the following propositon.

\begin{prop}\label{prop:BelDiffConvergentTo1}
  For any $\Phi_0 \in \SQD(\sigma)$, let $\{\rho_t\}_{t>0}$ be the harmonic maps ray in the direction of $\Phi_0$. Then, for every nonzero point $p$ of $\Phi_0$, we have
  \begin{equation*}
    |\nu(t)(p)|^2 \nearrow 1\qquad (\text{as}\ t\to\infty).
  \end{equation*}
\end{prop}

\begin{remark}
	Let $M$ be the domain surface $(S,\sigma)$. In the case of closed surfaces, Wolf \cite{wolf1989teichmuller} proved this propositon in the following three steps.
	\begin{enumerate}[Step 1:]
		\item Show that $|\nu(t)(p)|^2$ converges to $1$ at alomost everywhere on $M$.
		\item Show that $(|\nu(t)(p)|^2)'>0$ on $M-\{\Phi_0(p)=0\}$.
		\item Exclude the possibility that $|\nu(t)(p)|^2\to\delta\neq 1$ as $t\to\infty$ for a nonzero point $p$ of $\Phi_0$.
	\end{enumerate}
	If the surface $S$ has punctures, we can proceed Step 1 and Step 3 in essentially the same way as the paper \cite{wolf1989teichmuller}. Therefore we give the proof of Step 2 here. Since the compactness of the surface is critical for Step 2 in \cite{wolf1989teichmuller}, We prove this inequality for punctured surfaces.
\end{remark}

\begin{prop}
	For every nonzero point $p\in M-\{\Phi_0(q)=0\}$, the derivative $(|\nu(t)(p)|^2)'$ with respect to $t$ is positive.
\end{prop}

\begin{proof}
By the formula (II) in \Cref{prop:formulae}, we have
\begin{equation*}
	(H(t)L(t))'=H'(t)L(t)+L'(t)H(t)=\frac{2t|\Phi_0|^2}{\sigma^2}=\frac{2}{t}H(t)L(t).
\end{equation*}
Therefore, for every $p\in M-\{\Phi_0(p)=0\}$,
\begin{equation}\label{eq:QuotientIsConstant}
	\frac{H'(t)}{H(t)}(p) + \frac{L'(t)}{L(t)}(p)=\frac{2}{t}.
\end{equation}

Next, we introduce some results on the holomorphic energy functions $H(t)$. The following is shown by Lohkamp.
\begin{lemma}[\cite{lohkamp1991harmonic} Lemma 14]\label{lemma:LohkampIneaquarlity}
  For $t_1,t_2>0$,
  \begin{equation*}
			\min\left\{1,\frac{t_2}{t_1}\right\} \cdot H(t_1) \leq H(t_2) \leq \max\left\{1,\frac{t_2}{t_1}\right\}\cdot  H(t_1)\ \ \text{on}\ M.
  \end{equation*}
\end{lemma}

Using \cref{lemma:LohkampIneaquarlity}, we show the following.

\begin{lemma}\label{lemma:LessThan1t}
	For every $t>0$ and $p\in M$ , we have
	\begin{equation*}
		0\leq\frac{H'(t)}{H(t)}(p)< \frac{1}{t},
	\end{equation*}
	where $H'$ is a derivative with respect to $t$.
\end{lemma}

\begin{proof}
	Since the holomorphic energy $H(t)$ is positive, by \Cref{lemma:LohkampIneaquarlity},
	\begin{equation*}
		\min\left\{0, \log\frac{t_2}{t_1} \right\}\leq \log\frac{H(t_2)}{H(t_1)}\leq\max\left\{0,\log\frac{t_2}{t_1}\right\}.
	\end{equation*}
	Therefore, we obtain
	\begin{equation*}
		0 \leq \frac{\log H(t_2)- \log H(t_1)}{t_2-t_1} \leq  \frac{\log t_2 -\log t_1}{t_2-t_1} .
	\end{equation*}
	Setting $t_1 = t, t_2 = t + h$ and tending $h$ to $0$, we find
	\begin{equation}\label{eq:EqualityisStrict}
		0\leq \frac{H'(t)}{H(t)}\leq\frac{1}{t}.
	\end{equation}

	Next, we show that the the second inequality of \cref{eq:EqualityisStrict} is strict. If there exists a point $p_0 \in M$ such that $H'(t)/H(t)(p_0)=1/t$, since $p_0$ maximizes $H'(t)/H(t)$, we have
	\begin{equation*}
		0\geq\Delta_\sigma \frac{H'(t)}{H(t)}(p_0)=\Delta_\sigma(\log H(t))'=2(H'(t)-L'(t)).
	\end{equation*}
	Therefore, we find $H'(t)(p_0) \leq L'(t)(p_0)$.

	This implies that $p_0$ is not a zero of $\Phi_0$. If $p_0$ is a zero, by $H(t)>0$ and $H(t)L(t)=t^2|\Phi_0(p_0)|^2/\sigma^2=0$,
	\begin{equation*}
		L(t)(p_0)= 0
	\end{equation*}
	holds for any $t>0$. Hence we have $H'(t)(p_0)\leq0$, however this is impossible since $H(t)(p_0)>0$ and $H'(t)/H(t)(p_0)=1/t$. Thus, we find that $\Phi_0(p_0)\neq 0$.

	By $H(t)-L(t)>0$, we have
	\begin{equation*}
		\frac{L'(t)}{L(t)}(p_0)\geq \frac{H'(t)}{L(t)}(p_0) > \frac{H'(t)}{H(t)}(p_0).
	\end{equation*}
	However, this contradicts $H'(t)/H(t)(p_0)=1/t$ and \Cref{eq:QuotientIsConstant}.
\end{proof}
By \Cref{lemma:LessThan1t} and \Cref{eq:QuotientIsConstant}, we have
\begin{equation*}
	\frac{L'(t)}{L(t)}> \frac{H'(t)}{H(t)}\ \  \text{on}\  M-\{\Phi_0(p)=0\}.
\end{equation*}
Thus, we obtain
\begin{equation*}
	(|\nu(t)|^2)'=\left(\frac{L(t)}{H(t)}\right)'=\frac{L'(t)H(t)-L(t)H'(t)}{H(t)^2}>0\ \  \text{on}\  M-\{\Phi_0(p)=0\}.
\end{equation*}
This complets the proof.
\end{proof}

\subsection{Asymptotic length of horizontal and vertical arcs}

Given hyperbolic structures $\sigma,\rho$ on $S$, we can isotope $\rho $ so that the identity map $\id\colon(S,\sigma)\to(S,\rho)$ is harmonic. For the rest of this paper, we always take such a representative in its isotopy class.

A holomorphic quadratic differential $\Phi$ defines two measured foliations on $M$ which are orhthogonal to each other away from the zeros. By changing a conformal coordinate $z$ on $M$ to the natural coordinate $\zeta=\xi+i\eta$ of $\Phi$, the representation of $\Phi$ in terms of $\zeta$ is identically $1$ away from zero, that is
\begin{equation*}
	\Phi\,dz^2 = d\zeta^2.
\end{equation*}
Then lines parallel to $\xi$-axis (resp.\,$\eta$-axis) define a singular foliation on $M$ such that its singular points are zeros of $\Phi_0$, and also $|d\eta|$ (resp.\,$|d\xi|$) defines a transverse measure for the singular foliation.
We call the measured foliation the \textit{horizontal} (resp.\textit{vertical}) \textit{foliation} of $\Phi_0$, and it is denoted by $F_h(\Phi)$ (resp.\,$F_v(\Phi)$).
If $\Phi_0$ has a pole of digree one at a puncture on $S$, the foliation is one-pronged at the puncture. An arc along a leaf of horizontal (resp.vertical) measured foliation of $\Phi$ is called a \textit{horizontal} (resp.\textit{vertical}) \textit{arc} of $\Phi_0$.

Let $\Phi_0\in\SQD(\sigma)$ and $\{\rho_t\}$ be the harmonic maps ray in the direction of $\Phi_0$. In this subsection, we consider the asymptotic $\rho_t$-length of horizontal and vertical arcs of $\Phi_0$. Here, we recall Wolf's setting in \cite{wolf1989teichmuller}.
For the ray $\{\rho_t\}$, we define conformal coordinates $z=x+iy$ on $M$ such that $\pdv{x}$ and $\pdv{y}$ give an orthonormal frame field on $M$ and also they are respectively maximum and minimum streching directions of the differential map $dh(t)$.
Then, they are tangent to the horizontal and vertical foliations of $\Phi_0$, respectively. The conformal coordinates $z$ are defined from away zeros of $\Phi_0$. By the definitions of $\pdv{x}$ and $\pdv{y}$, they are orthogonal to each other in the $\rho_t$-metric. Thus, we have
\begin{equation*}
	t\Phi_0 \,dz^2= \frac14 \left( \left\|\pdv{x}\right\|^2_{\rho_t}-\left\|\pdv{y}\right\|^2_{\rho_t} \right) \, dz^2.
\end{equation*}
Therefore, the $|\Phi(\rho_t)|$ metric length of a tangent vector $\pdv{x}$ is
\begin{equation}\label{eq:Dilatation1}
	\left\|\pdv{x}\right\|_{\Phi(\rho_t)}^2=\frac14 \left( \left\|\pdv{x}\right\|^2_{\rho_t}-\left\|\pdv{y}\right\|^2_{\rho_t} \right).
\end{equation}
Moreover, by computation, we find
\begin{equation}\label{eq:Dilatation2}
	|\nu(t)|=\frac{1-\|\pdv{y}\|_{\rho_t}/\|\pdv{x}\|_{\rho_t}}{1+\|\pdv{y}\|_{\rho_t}/\|\pdv{x}\|_{\rho_t}}.
\end{equation}
Hence, by \Cref{prop:BelDiffConvergentTo1}, the following holds.
\begin{prop}\label{prop:NormConvergence}
	Let $\{\rho_t\}_{t>0}$ be the harmonic maps ray in the direction $\Phi_0$. Then
	\begin{enumerate}
		\item for every $p$ with $\Phi_0(p)\neq 0$,
		$\|\pdv{y}\|_{\rho_t}/\|\pdv{x}\|_{\rho_t}\searrow 0 $ as $t\to\infty$, and
		\item for every $p$ with $\Phi_0(p)\neq 0$,
		$\|\pdv{x}\|_{|4\Phi(\rho_t)|}/\|\pdv{x}\|_{\rho_t}\nearrow 1 $ as $t\to\infty$.
	\end{enumerate}
\end{prop}

\begin{prop}\label{prop:LengthConvergence}
	Let $\{\rho_t\}_{t>0}$ be a harmonic maps ray in the direction $\Phi_0$. For an arc $\gamma$ on $M$, let $\ell_{\rho_t}(\gamma)$ denote the $\rho_t$-length of $\gamma$.
	\begin{enumerate}
		\item Let $\gamma$ be a compact horizontal arc of $\Phi_0$ containing no zeros of $\Phi_0$. Then, there exist constants $c_0, c_1>0$ depending only on $\gamma$ such that for every $t>1$
		\begin{equation*}
			0<c_0<\ell_{\rho_t}(\gamma)t^{-1/2} <c_1<\infty.
		\end{equation*}
		\item Let $\gamma$ be a compact vertical arc of $\Phi_0$ (that may contain zeros of $\Phi_0$). Then
		\begin{equation*}
			\ell_{\rho_t}(\gamma) t^{-1/2} \to 0\ \ \text{as}\ t\to\infty.
		\end{equation*}
	\end{enumerate}
\end{prop}

\begin{proof}
	This proposition can be shown similarly to \cite{wolf1989teichmuller}, using the fact that the holomorphic energy function is bounded \cite[Corollary 3 and Lemma 9]{lohkamp1991harmonic} instead of the compactness for a closed surface.
\end{proof}

Following lemma for closed surfaces is proved by Minsky in \cite[Lemma 3.3]{minsky1992harmonic}. Since Minsky's proof works for \Cref{lemma:exponentialestimate}, we omit the proof.

\begin{lemma}\label{lemma:exponentialestimate}
	Let $\Phi\in\QD(\sigma)$. Suppose that $p\in M$ is at least distance $d$ away from poles and zeros of $\Phi$, and
	\begin{equation*}
		\log|\nu(\Phi)|^{-1} < b \ \ \ \text{on}\ B_{|\Phi|}(p,d),
	\end{equation*}
	where $B_{|\Phi|}(p,d)$ denotes the $|\Phi|$-radius $d$ disk centred at $p$, and $\nu(\Phi)=\nu(\rho(\Phi))$. Then
	\begin{equation*}
		\log|\nu(\Phi)(p)|^{-1} < \frac{b}{\cosh d}\,.
	\end{equation*}
\end{lemma}

By using \cref{lemma:exponentialestimate} and some estimates in \cite{wolf1991high}, we can bound the asymptotic $\rho_t$-length of horizontal arcs of $\Phi_0$ from above by the $|\Phi_0|$-length.

\begin{prop}\label{prop:ExponentialEstimateHoriArc}
	Let $\Phi_0 \in \SQD(\sigma)$, and let $\{\rho_t\}_{t>0}$ be the harmonic maps ray in the direction of $\Phi_0$. If $\gamma$ be a compact horizontal arc of $\Phi_0$ containing no zeros, then there exist $C,D>0$ depending only on $\gamma$ such that
	\begin{equation*}
		\ell_{\rho_t}(\gamma)<t^{1/2}\ell_{|4\Phi_0|}(\gamma)(1+Ce^{-D\sqrt{t}})\ \text{for all}\ t>0.
	\end{equation*}
\end{prop}

\begin{proof}
	Let $\Sigma(\Phi_0)$ denote the set consisting of punctures and zeros of $\Phi_0$. We set
	\begin{equation*}
		d(t)=\inf \{d_{t|\Phi_0|}(p,q)\mid p\in\gamma, q\in\Sigma(\Phi_0)\}.
	\end{equation*}
	Then we find that $d(t)=t^{1/2}d(1)$ by the definition of quadratic differential metrics. We define a constant $C_0$ as
	\begin{equation*}
		C_0=\sup\{ \log|\nu(1)(q)|^{-1}\mid q\in N_{|\Phi_0|}(\gamma,d(1)/2) \},
	\end{equation*}
	where $N_{|\Phi_0|}(\gamma,d(1)/2) $ denote the $(d(1)/2)$-neighborhood of $\gamma$ in the $|\Phi_0|$ metric. Since $|\nu(t)|$ increases monotonically in $t$, we see that for every $q\in N_{|\Phi_0|}(\gamma,d(1)/2)$ and every $t>1$,
	\begin{equation*}
		\log|\nu(t)(q)|^{-1}<\log|\nu(1)(q)|^{-1}\leq C_0.
	\end{equation*}
	Therefore, by \cref{lemma:exponentialestimate},
	\begin{equation*}
		\log |\nu(t)|^{-1}<\frac{C_0}{\cosh(d(t)/2)}=\frac{C_0}{\cosh(t^{1/2}d(1)/2)}<C_0e^{-\sqrt{t}d(1)/2}\ \
	\end{equation*}
	holds on $\gamma$. Then, we have
	\begin{align*}
		|\nu(t)|^{-1/2}-1&<\frac{1}{|\nu(1)|^{1/2}}(1-|\nu(t)|^2)\\
		&\leq\frac{1}{\min_\gamma|\nu(1)|^{1/2}}\log|\nu(t)|^{-2} \\
		&\leq\frac{2}{\min_\gamma|\nu(1)|^{1/2}}C_0e^{-\sqrt{t}d(1)/2}.
	\end{align*}
	Setting
	\begin{equation*}
		C=\frac{2}{\min_\gamma|\nu(1)|^{1/2}}C_0\ \ \text{and}\ \  D=d(1)/2,
	\end{equation*}
	we have
	\begin{align*}
		\ell_{\rho_t}(\gamma)&=\int_\gamma\left\| \pdv{x} \right\|_{\rho_t}ds\\
		&=\int_{\gamma}\{H(t)^{1/2}+L(t)^{1/2}\}\,ds_\sigma\\
		&=\int_{\gamma}H(t)^{1/2}(1+|\nu(t)|)\,ds_\sigma\\
		&=\int_\gamma\frac{t^{1/2}|\Phi_0|^{1/2}}{|\nu(t)|^{1/2}}(1+|\nu(t)|)\frac{ds_\sigma}{\sigma^{1/2}}\\
		&=t^{1/2}\int_\gamma(1+(|\nu(t)|^{-1/2}-1))(2-(1-|\nu(t)|))\,ds_{|\Phi_0|}\\
		&<2t^{1/2}\int_\gamma(1+(|\nu(t)|^{-1/2}-1))\,ds_{|\Phi_0|}\\
		&<t^{1/2}\ell_{|4\Phi_0|}(\gamma)(1+Ce^{-D\sqrt{t}}),
	\end{align*}
	as $2\ell_{|\Phi_0|}(\gamma)=\ell_{|4\Phi_0|}(\gamma)$. We obtain the desired inequality.
\end{proof}

\begin{corollary}
	Under the assumptions of \Cref{prop:ExponentialEstimateHoriArc},
	\begin{equation*}
		\lim_{t\to\infty}t^{-1/2}\ell_{\rho_t}(\gamma)\leq\ell_{|4\Phi_0|}(\gamma)
	\end{equation*}
	holds.
\end{corollary}

The following proposition will be used in the proof of \Cref{lemma:StaircaseReoresentative}.

\begin{lemma}\label{lemma:Rneighborhood}
	Let $q\in \overline{S}-S$, a puncture of $S$. Then for every $\Phi\in\QD(\sigma)$,
	\begin{equation*}
		\ell_{|\Phi|}(\partial B_{|\Phi|}(q,R))\leq L_1R,
	\end{equation*}
	where $L_1$ is a constant which depends only on the topological type of $S$.
\end{lemma}

\begin{remark}
	Note for closed Riemann surfaces that \cref{lemma:Rneighborhood} is a special case of Lemma 4.1 in \cite{minsky1992harmonic}. For every point $p\in\overline{S}$, let $\deg(p)$ denote the degree of $\Phi$ at $p$. We know that a singularity $p$ has a cone angle of $(n+2)\pi$, or concentrated curvature $-n\pi$, where $n=\deg (p)$. Note that, if $p$ is a pole of $\Phi$, then $\deg(p)=-1$.
\end{remark}

\begin{proof}
	For a number $r$ with $0\leq r\leq R$, we set $\gamma_r=\partial B_{|\Phi|}(q,R)$. Then, we have
	\begin{equation*}
		\frac{d}{dr}\ell_{|\Phi|}(\gamma_r)=\kappa (\gamma_r),
	\end{equation*}
	where $\kappa(\gamma_r)$ is a total curvature of $\gamma_r$ in the $|\Phi|$ metiric. By the Gauss-Bonnet theorem, we have
	\begin{align*}
		\kappa(\gamma_r)&=2\pi\chi(B_{|\Phi|}(q,r))-\sum_{p\in B_{|\Phi|}(q,R)}(-\deg(p)\pi)\\
		&\leq2\pi+(4g-4+n)\pi.
	\end{align*}
	Therefore, $\ell_{|\Phi|}(\gamma_R)=\int_0^R\kappa(\gamma_r)\,dr \leq(4g-2+n)\pi R$
	follows.
\end{proof}

\section{The identification of the compactifications}

\subsection{The Fundamental Lemma}

As described at the begginning of the previous section, a holomorphic quadratic differential $\Phi$ on $M$ defines the vertical measured foliation on $M$.
In fact, by the Hubbard-Masur theorem, the map $\Phi \mapsto F_{v}(\Phi)$ is a homeomorphism from $\QD(\sigma)$ to $\MF$ (\cite{hubbard1979quadratic}, also see \cite[p.206]{gardiner1987teichmuller}), where $\MF$ denotes the set of measured foliations on $S$ (here $\MF$ contains the empty measured foliation).
We define a homeomorphism $\beta\colon\QD(\sigma)\to\MF$ by $\beta\Phi= F_{v}(4\Phi)$. Then, by the definition of the transverse measure of $\beta\Phi$, we find that, for every $[\gamma]\in\mathcal{C}$
\begin{equation*}
	i(\beta\Phi,[\gamma])=\|\Phi\|^{1/2}i(\beta(\Phi/\|\Phi\|),[\gamma]).
\end{equation*}

 The main purpose of this subsection is to prove the following proposition analogous to \cite[Lemma 4.6]{wolf1989teichmuller}.

 \begin{theorem}\label{prop:FundamentalInequality}
	 Fix $[\gamma]\in\mathcal{C}$ and $\varepsilon>0$. Then, for every $\rho\in\Teich_{g,n}$, there exist nonnegative constants $c_0=c_0([\gamma],\varepsilon)$, $k_0=k_0(\varepsilon,\|\Phi(\rho)\|)$ and $\eta=\eta([\gamma],\varepsilon,\|\Phi(\rho)\|)$ such that
 	\begin{equation*}
 		i(\beta\Phi(\rho),[\gamma])\leq\ell_\rho([\gamma])\leq k_0i(\beta\Phi(\rho),[\gamma])+\eta,
 	\end{equation*}
 	and
 	\begin{equation*}
 		k_0\searrow 1,\ \ \eta\|\Phi(\rho)\|^{-1/2}\to c_0 <\varepsilon
 		\ \ \text{as}\  \|\Phi(\rho)\|\to\infty.
 	\end{equation*}
\end{theorem}

 \begin{remark}
 	In the statement of Lemma 4.6 in \cite{wolf1989teichmuller}, it is written that the constant $k_0$ depends on $[\gamma]$. However, from his proof, one can observe it is actually independent.
 \end{remark}

\begin{proof}
	The lower bound $i(\beta\Phi(\rho),[\gamma])\leq\ell_\rho([\gamma])$ is shown in the same manner with \cite{wolf1989teichmuller}, so we omit the proof here.

	In order to show the upper bound, we first show the following lemma.

	\begin{prop}\label{lemma:StaircaseReoresentative}
		For each $\Phi \in \SQD(\sigma)$, we can construct a representative $\gamma_\Phi\in[\gamma]$ so that there exist domains $R_0=R_0([\gamma],\varepsilon)$ and $R_1=R_1(\varepsilon)$ containing all punctures of $M$,  constants $\delta=\delta(\varepsilon)>0, K=K([\gamma])>0$, and positive integers $L=L([\gamma]), m=m([\gamma],\varepsilon)$ satisfying the following conditions:
		\begin{enumerate}
			\item The curve $\gamma_\Phi$ consists of horizontal and vertical arcs of $\Phi$ and does not intersect with the neighborhood $R_0=R_0([\gamma],\varepsilon)$ of the punctures of $M$.
			\item The horizontal arcs of $\gamma_\Phi$ are divided into the main part $\gamma_\Phi^h$ and the additional part $\widetilde{\gamma}_\Phi^h$ such that
			\begin{itemize}
				\item the main part $\gamma_\Phi^h$ is disjoint from
				\begin{equation*}
				\left(	\bigcup_{p} B_{\sigma}(p,\delta) \right) \cup R_1,
				\end{equation*}
				where the union is over all zeros $p$ of $\Phi$.
				\item the number of the connected segments constituting the additional part $\widetilde{\gamma}_\Phi^h$ is at most $m=m([\gamma],\varepsilon)$, and they are disjoint from some neighborhood of zeros of $\Phi$.
			\end{itemize}
			\item $i(\beta\Phi,\gamma_\Phi^h)=i(\beta\Phi,[\gamma])$ and $i(\beta\Phi,\widetilde{\gamma}_\Phi^h)<\varepsilon$.
			\item The total $|\Phi|$-length of the vertical arcs of $\gamma_\Phi$ is uniformly bounded by $K=K([\gamma])$, i.e. $i(\beta(-\Phi),\gamma_\Phi^v)<K$, where $\gamma_\Phi^v$ is the union of the vertical arcs. Moreover, the number of the connected segments of the vertical arcs which contain a zero of $\Phi$ is at most $L=L([\gamma])$.
		\end{enumerate}
	\end{prop}

	\begin{proof}
		We fix $\Phi_0\in\SQD(\sigma)$. By the compactness of $\SQD(\sigma)$, it is enough to show that the claim holds on a neighborhood of $\Phi_0$.

		For a quadratic differential $\Phi \in \SQD(\sigma)$, we let
		\begin{align*}
			Z(\Phi)&=\{\text{zeros of }\Phi\text{ in }M\}, \\
			\Preg(\Phi)&=\{\text{punctures which are regular points or zeros of } \Phi\},\\
			\Ppole(\Phi)&=\{\text{punctures which are poles of }\Phi\}, \\
			\Sigma(\Phi)&= Z(\Phi)\cup\Preg(\Phi)\cup\Ppole(\Phi),  \text{the singular set of }\Phi.
		\end{align*}
		Furthermore, we pick a positive constant $\delta'=\delta'(\Phi_0)$ such that, for all distinct $p_i, p_j \in \Sigma(\Phi_0)$, the $|\Phi_0|$-distance between $B_{|\Phi_0|}(p_i,2\delta')$ and $B_{|\Phi_0|}(p_j,2\delta')$ is at least $2\delta'$, and also
		\begin{equation*}
			\delta'(\Phi_0) < \Bigl(\min_{\Phi\in\SQD(\sigma)} \inj|\Phi|\Bigr)/(8g-5+2n),
		\end{equation*}
		where we set $\inj|\Phi|\coloneqq\inf_{[\gamma]\in\mathcal{C}}\ell_{|\Phi|}([\gamma])/2$. We may assume  $\varepsilon<\delta'$, since $\varepsilon>0$ is arbitrary. Then, we pick a sufficiently small neighborhood $\mathcal{N}$ of $\Phi_0$ in $\SQD(\sigma)$ so that every $\Phi\in\mathcal{N}$ satisfies the following conditions:
		\begin{enumerate}
			\item The connected components of
			\begin{equation*}
				N(\Sigma(\Phi_0),\delta')\coloneqq \bigcup_{p\in\Sigma(\Phi_0)} B_{|\Phi_0|}(p,\delta')
			\end{equation*}
			bijectively correspond to the connected components of
			\begin{equation*}
				N(\Sigma(\Phi),\delta')\coloneqq \bigcup_{p\in\Sigma(\Phi)} B_{|\Phi|}(p,\delta')
			\end{equation*}
			by the correspondence between $\Sigma(\Phi)$ and $\Sigma(\Phi_0)$.
			\item If we fill the punctures, every connected component $C$ of $N(\Sigma(\Phi),\delta')$  is a topologically disk, and the $|\Phi|$-distance between $C$ and the other connected components  is at least $\delta'$.
			\item The total length of critical vertical leaves of $\Phi$ contained in $N(\Sigma(\Phi),\delta')$ is uniformly bounded from above by a constant $K_1=K_1(\Phi_0)$.
			\item The zeros of $\Phi$ splitting off (see Figure \ref{fig:SplittingZeros}) from $q\in\Preg(\Phi_0)$ is contained in $B_{|\Phi|}(q,r/2)$, where $r$ is a sufficiently small constant which depends only on $[\gamma]$, $\varepsilon$ and $\Phi_0$ and is defined later in \cref{eq:definitionofr}.
		\end{enumerate}
		\begin{figure}[t]
			\centering
			\includegraphics[scale=1.5]{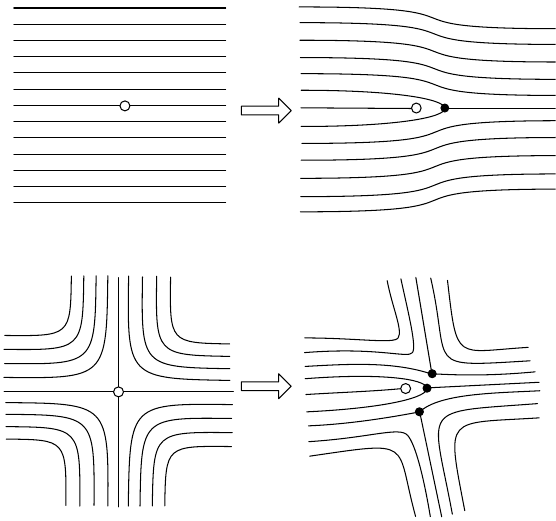}
			\caption{The zeros splits off from a puncture.}
			\label{fig:SplittingZeros}
		\end{figure}

		Under the preparation, we describe the construction of the representative $\gamma_\Phi$ for each $\Phi\in\mathcal{N}$. First, we beggin with the $|\Phi|$-geodesic representative $\Gamma_\Phi$ of $[\gamma]$.
		Note that $\Gamma_\Phi$ may touch some punctures and cannot be realized on $M$ in a strict sense (see Figure \ref{figure:GeodesicRepre}).
		(If all of punctures of $S$ are simple poles of $\Phi$, there exists a strict $|\Phi|$-geodesic representative  of $[\gamma]$ in $M$, see \cite{syau1996existence}.)
		\begin{figure}[t]
			\centering
			\includegraphics[scale=1.2]{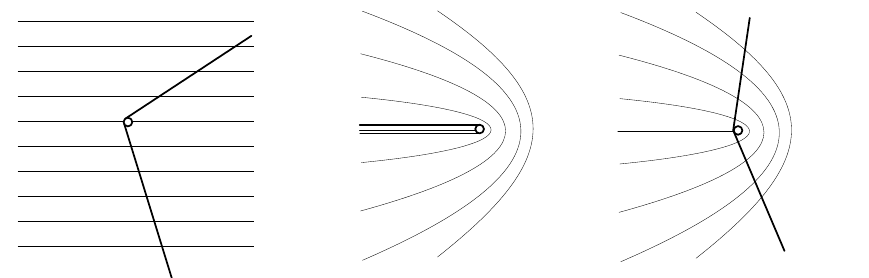}
			\caption{Left and middle: we illustrate examples of the geodesic representative in a slightly broad sense around a puncture. Right: there is no such a representative in the right of the figures.}
			\label{figure:GeodesicRepre}
		\end{figure}
		Each segment of the geodesic representative $\Gamma_\Phi$  outside of $N(\Sigma(\Phi),\delta')$ is a Euclidean straight line segment, so we replace such a straight segment of $\Gamma_\Phi$ with a $\Phi$-\textit{staircase curve} which is union of horizontal and vertical arcs of $\Phi$.
		Let $\Gamma_\Phi'$ be the resulting curve. Then
		\begin{equation*}
			i(\beta\Phi,\Gamma'_\Phi)=i(\beta\Phi,[\gamma])\ \text{and}\ i(\beta(-\Phi),\Gamma'_\Phi)<\max_{\Phi\in\SQD(\sigma)} \ell_{|4\Phi|}([\gamma]).
		\end{equation*}

		Next, let $C$ be a connected component of $N(\Sigma(\Phi),\delta')$. Then, the number of punctures contained in $C$ is at most one.

		\begin{claim*}
			The number of connected components of $\Gamma_\Phi'\cap C$ is, at most, $\alpha=\alpha([\gamma],\delta')$.
		\end{claim*}

		\begin{proof}
			 Case I. We suppose that $C$ contains a puncture $q$. Then, $C$ contains, at most, $4g-4+n$ zeros of $\Phi$. Therefore, $C$ is contained in the disk
			\begin{equation}\label{eq:DefinitionOfD}
				D\coloneqq B_{|\Phi|}(q,(8g-7+2n)\delta').
			\end{equation}
			The disk $D$ does not contain the other puncture and $D$ is embedded into $M$, since
			$(8g-7+2n)\delta'<\inj|\Phi|$. Let $p$ be a zero of $\Phi$ contained in $C$. Let $\Gamma_\Phi''$ denote a connected component of $\Gamma_\Phi\cap D$.
			Notice that $B_{|\Phi|}(p,\delta')$ may not be a convex set, since it can contain a pole of $\Phi$ at the puncture $q$. However, considering a double branched covering of $D$, we find that the number of the connected components of $\Gamma_\Phi''\cap B_{|\Phi|}(p,\delta')$ is at most two (see Figure \ref{fig:DoubleBranched}).
			\begin{figure}
				\centering
				\begin{overpic}
					[scale=1.3]{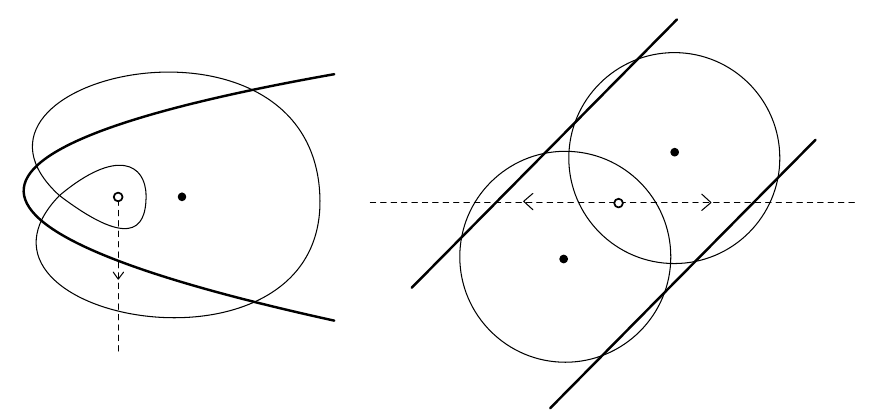}
					\put(14,26){$q$}
					\put(22,26){$p$}
					\put(68,26){$\tilde{q}$}
					\put(61.5,17){$\tilde{p}$}
					\put(79,32){$\tilde{p}$}
				\end{overpic}
				\caption{Left: it is difficult to directly understand the $\delta'$-neighborhood of a zero near a pole of $\Phi$. Right: the lift of the double cover branched at the puncture $q$.}
				\label{fig:DoubleBranched}
			\end{figure}
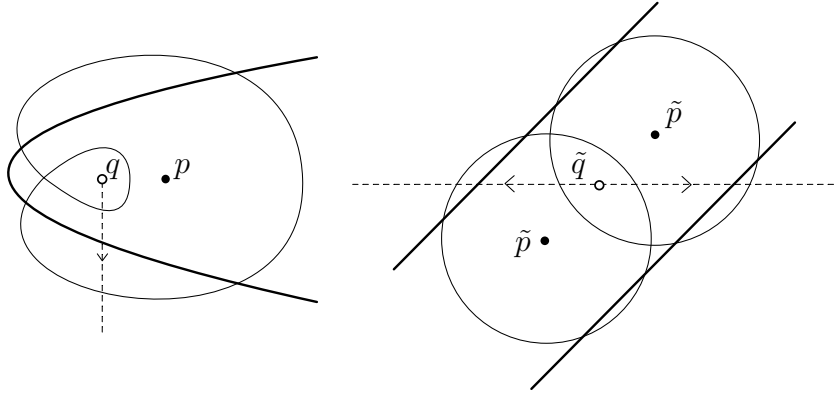
			Thus, the number of connected components of $\Gamma_\Phi''\cap C$ is at most $2(4g-3+n)$, where $4g-3+n$ is the maximum of the number of the disks constituting $C$. Let $\gamma'$ denote a subarc  of $\Gamma_\Phi$ which leaves $D$ and comes back to $D$.
			The endpoints of $\gamma'$ is on $\partial D$.
			Then, since $D$ is convex and $\Gamma_\Phi$ is geodesic, $\gamma'$ is not isotopic to $\partial D$ rel the endpoints.
			Thus, we see that the $|\Phi|$-length of $\gamma'$ is at least $\delta'$.
			By the above discussion, we find that, every time $\Gamma_\Phi$ intersects $D$, the number of the connected components of $C\cap\Gamma_\Phi$ increases by at most $8g-6+2n$ and also the number of the connected components of $\Gamma_\Phi\cap D$ is at most
			\begin{equation}\label{eq:boundednumber}
				\biggl\lceil\max_{\Phi\in\SQD(\sigma)}\{\ell_{|4\Phi|}([\gamma])\}/\delta'(\Phi_0)\biggr\rceil.
			\end{equation}

			Case II. Suppose next that $C$ contains no punctures.
			Then, $C$ contains, at most $4g-4+n$, zeros of $\Phi$.
		 	Therefore, the $|\Phi|$-diamiter of $C$ is, at most, $(8g-8+2n)\delta'$.
			Hence, $C$ is contained in a ball $B= B_{|\Phi|}(p_0,(4g-4+n)\delta')$ for a point $p_0\in M$. Then the ball contains at most one puncture, since $(8g-8+2n)\delta'<(8g-5+2n)\delta'<\inj|\Phi|$.

			Then we have two cases.
			First, suppose that $B$ contains a puncture $q$. Then
			\begin{equation*}
				D= B_{|\Phi|}(q,(8g-7+2n)\delta')
			\end{equation*}
			contains $B$, and $D$ contains no other punctures. Therefore $D$ is convex, and thus we can apply the argument of Case I to $D$.
			However, $B_{|\Phi|}(p,\delta')$ is now convex for a zero $p$ of $\Phi$ contained in $C$, so we do not need to take a double branched covering.
			Thus we obtain the desired upper bounds.
			Second, suppose that $B$ contains no punctures. If
			the neighborhood $N_{|\Phi|}(B,\delta')$ of $B$ contains a puncture $q$, $B$ is contained in the disk
			\begin{equation*}
				D= B_{|\Phi|}(q,(8g-7+2n)\delta').
			\end{equation*}
			Therefore we can show the desired claim as in the first case that $C$ contains a puncture $q$.
			If the neighborhood $N_{|\Phi|}(B,\delta')$ of $B$ does not contain a puncture, each connected component of $\Gamma_\Phi\setminus B$ has length at least $\delta'$.
			Then, setting $D=B$, we again apply the argument of Case I to $D$.
 		\end{proof}
		Let $\overline{\gamma}$ denote a connected component of $\Gamma_\Phi'\cap C$. If a vertical arc of $\Phi$ contains a singular point of $\Phi$ as its end point, then it is called a \textit{critical vertical arc}. Then, either (i) $\overline{\gamma}$ has at most one intersection with each critical vertical arc contained in $C$, or (ii) $\overline{\gamma}$ is a union of critical vertical arcs.
		We leave the curves in the case (ii), and deform the curves in the case (i). The number of the critical vertical arcs contained in $C$ is, at most, $3(4g-4+n)+1$, so the number of the intersections of $\overline{\gamma}$ and the critical vertical arcs is clearly finite. Then, we divide $\overline{\gamma}$ at these intersection points into finitely many segments. For each such segment of $\overline{\gamma}$, we replace it with a curve consisting of horizontal arcs and vertical arcs of $\Phi$. Then, we drag the horizontal measure out of $C$ by adding vertical arcs (Figure \ref{fig:DraggingOut}).
		\begin{figure}[t]
			\centering
			\includegraphics[scale=1.2]{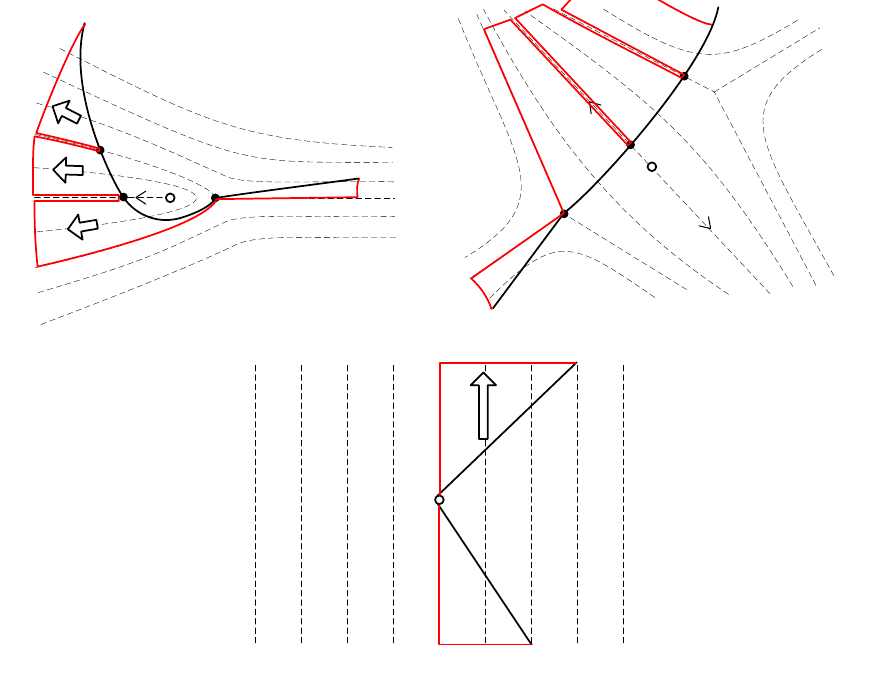}
			\caption{The dotted lines denote vertical leaves of $\Phi$. The black points denote intersections of critical vertical arcs and $\overline{\gamma}$. The upper right figure illustrates the lift of the foliation in the upper left figure.}
			\label{fig:DraggingOut}
		\end{figure}
		Then we set
		\begin{equation*}
			R_1=\bigcap_{\Phi\in\mathcal{N}}\left(\bigcup_{q\in\overline{S}-S}B_{|\Phi|}(q,\delta')\right).
		\end{equation*}
		Let $\widehat{\gamma}$ denote the resulting curve from $\overline{\gamma}$.

		Next, we add the following operation, which is denoted by $(\ast)$, for each the connected component $\alpha$ of $\widehat{\gamma}\cap B_{|\Phi|}(q,r)$:
		\begin{enumerate}
			\item isotope $\alpha$ to $\partial B_{|\Phi|}(q,r)$ rel the endpoints, and
			\item then further isotope $\alpha$ on $\partial B_{|\Phi|}(q,r)$ to a union of vertical and horizontal arcs tangent to $\partial B_{|\Phi|}(q,r)$ (see Figure \ref{fig:Rsetting}).
		\end{enumerate}
		By this operaton ($\ast$), $\gamma_\Phi$ does not enter the $r$-neighborhood of the puncture $q$. Since, by \Cref{lemma:Rneighborhood}, the $|\Phi|$-length of $\partial B_{|\Phi|}(q,r)$ is at most $L_1 r$, this operation $(\ast)$ increases the horizontal measure by at most $L_1 r$. Thus, if we take $r$ so that
		\begin{equation}\label{eq:definitionofr}
			m_1 \alpha L_1 r <\varepsilon,
		\end{equation}
		then the increase due to this operation ($\ast$) is less than $\varepsilon$, where $m_1$ is the upper bounds on the total number of the operation $(\ast)$ all over each $\widehat{\gamma}$ and $\alpha$ is as in the above claim. Moreover, $m_1$ depends only on the topology of $M$. Thus, we set a domain $R_0$ by
		\begin{equation*}
			R_0=\bigcap_{\Phi\in\mathcal{N}}\left(\bigcup_{q\in\overline{S}-S}B_{|\Phi|}(q,r)\right)
		\end{equation*}
		and define $\gamma_\Phi$ as the resulting representative of $[\gamma]$. The constant $\delta$ and the domain $R_1$ in \Cref{lemma:StaircaseReoresentative} is determined by $\delta'$.
		\begin{figure}
			\centering
			\includegraphics[scale=1.2]{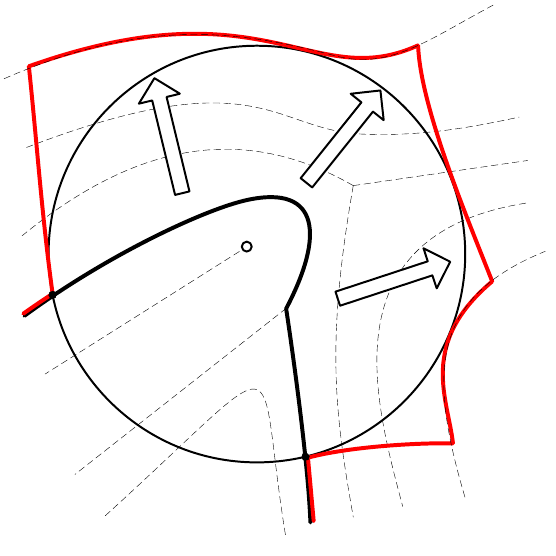}
			\caption{This is an example of a vertical arc in $B_{|\Phi|}(q,r)$. By the operation $(\ast)$, we let the curve not enter $B_{|\Phi|}(q,r)$.}
			\label{fig:Rsetting}
		\end{figure}
	\end{proof}

	Next, we describe the constant $k_0$. For $\rho\in\Teich_{g,n}$, we set
	\begin{equation*}
		M'= M- R_1,\ M_{\delta}(\rho)= M'-\bigcup_{p\in Z(\Phi)} B_\sigma(p,\delta).
	\end{equation*}
	Furthermore, we define a function $k_2$ by
	\begin{equation*}
		k_2(\rho)= \max_{p\in M_{\delta(\rho)}}\frac{\|(\pdv{x})_p\|_{\rho}}{\|(\pdv{x})_p\|_{|4\Phi(\rho)|}}.
	\end{equation*}

	\begin{claim*}
		The function $k_2$ is upper semicontinuous in $\Teich_{g,n}$.
	\end{claim*}

	\begin{proof}
		This proof is similar to that in \cite{wolf1989teichmuller}. \cite[p.464]{wolf1989teichmuller}.
	\end{proof}

	Hence, defining a function $\kappa_A:\SQD(\sigma) \to\R_{\geq0}$ for $A>0$ by
	\begin{equation*}
		\kappa_A(\Phi_0)\coloneqq k_2(\rho(A\Phi_0)),
	\end{equation*}
	we find that $\kappa_A$ is the upper semicontinuous function on $\SQD(\sigma)$. Since a upper semicontinuous function on a compact set has a maximum, we set
	\begin{equation*}
		k_0(A,\varepsilon)= \max_{\Phi_0\in\SQD(\sigma)}\kappa_A(\Phi_0).
	\end{equation*}
	Then, by \Cref{prop:NormConvergence}, for every $p\in M_\delta(\Phi_0)$,
	\begin{equation*}
		\frac{\|(\pdv{x})_p\|_{\rho(A\Phi_0)}}{\|(\pdv{x})_p\|_{|4A\Phi_0|}} \searrow 1\ \ \text{as }A\to\infty.
	\end{equation*}
	Therefore $\kappa_A(\Phi_0)\searrow 1$ as $A\to\infty$. Since $\{\kappa_A\}$ converges uniformly on $\SQD(\sigma)$ by Dini's theorem, we find that $k_0\searrow 1$ as $A\to\infty$. Thus we have
	\begin{align*}
		\ell_\rho(\gamma^h_{\Phi_0(\rho)})
		&=\sum_{\gamma\subset\gamma^h_{\Phi_0}} \int_{\gamma}\left\|\pdv{x}\right\|_{\rho}ds \\
		&\leq\sup_{\gamma_{\Phi_0}^h}\frac{\|(\pdv{x})_p\|_{\rho}}{\|(\pdv{x})_p\|_{|4\Phi(\rho)|}}\sum\int_{\gamma}\left\|\pdv{x}\right\|_{|4\Phi(\rho)|}ds\\
		&\leq k_2(\rho)i(\beta\Phi(\rho),\gamma_{\Phi_0}^h)\\
		&\leq k_0(A,\varepsilon) i(\beta\Phi(\rho),\gamma_{\Phi_0}^h),
	\end{align*}
	where $\Phi_0(\rho)\coloneqq\Phi(\rho)/\|\Phi(\rho)\|$.

	Finally, we describe the way of setting the constant $\eta=\eta([\gamma],\varepsilon,A)$. By \Cref{prop:LengthConvergence}, for any $\Phi_0\in\SQD(\sigma)$
	\begin{equation*}
		\ell_{\rho(A\Phi_0)}(\gamma_{\Phi_0}^v)A^{-1/2}\to0\ \ \text{as } A\to\infty.
	\end{equation*}
	Since $\gamma_{\Phi_0}$ continuously changes with respect to $\Phi_0$ and $\gamma_{\Phi_0}^v$ contains at most finite zeros of $\Phi_0$, the total $\sigma$-length of $\gamma_{\Phi_0}^v$ is bounded from above. Therefore, we conclude that there exists a constant $\eta_1([\gamma],A)$ such that
	\begin{equation*}
		\sum_{\gamma\subset\gamma_{\Phi_0(\rho)}^v}\ell_{\rho}(\gamma) < \eta_1\ \text{and\ \  }\eta_1A^{-1/2}\to0\ \ \ \text{as }A\to\infty.
	\end{equation*}
	By \Cref{prop:ExponentialEstimateHoriArc} and the construction of the representative $\gamma_\Phi\in[\gamma]$ for $\Phi\in\SQD(\sigma)$, we can take constants $C$ and $D$ (see \Cref{remark:ConstantEta2}) such that, for every $\Phi \in \SQD(\sigma)$,
	\begin{equation}\label{eq:CandD}
		\ell_{\rho(t\Phi)}(\widetilde{\gamma}_\Phi^h)<t^{1/2}i(\beta\Phi,\widetilde{\gamma}_\Phi^h)(1+Ce^{-D\sqrt{t}}).
	\end{equation}
	Therefore, if we set
	\begin{equation*}
		c_0=c_0([\gamma],\varepsilon)=\sup_{\Phi\in\SQD(\sigma)}i(\beta\Phi,\widetilde{\gamma}_\Phi^h)\
	\end{equation*}
	and
	\begin{equation*}
		\eta_2=\eta_2([\gamma],\varepsilon,A)= A^{1/2}c_0(1+Ce^{-D\sqrt{t}}),
	\end{equation*}
	then, for every $\Phi \in \SQD(\sigma)$,
	\begin{equation*}
		\ell_{\rho(t\Phi)}(\widetilde{\gamma}_\Phi^h)<\eta_2\ \  \text{and}\ \ \eta_2A^{-1/2}\to c_0 <\varepsilon
	\end{equation*}
	as $A\to\infty$. Hence, setting $\eta\coloneqq\eta_1+\eta_2$, we have $\eta A^{-1/2}\to c_0$ as $A\to\infty$. Thus, for every $\rho\in\rho(A\cdot\SQD(\sigma))$,
	\begin{align*}
		\ell_\rho([\gamma])
			&\leq\ell_{\rho}(\gamma_{\Phi_0(\rho)})\\
			&=\ell_\rho(\gamma_{\Phi_0(\rho)}^h)+\ell_\rho(\widetilde{\gamma}_{\Phi_0(\rho)}^h) + \ell_{\rho}(\gamma_{\Phi_0(\rho)}^v)\\
			&<k_0i(\beta\Phi(\rho),[\gamma])+\eta,
	\end{align*}
	and the proof is complete.
\end{proof}

\begin{remark}\label{remark:ConstantEta2}
	Here, we describe the setting of $\eta_2$, namely the constants $C$ and $D$ in \Cref{eq:CandD}.
	In the proof of \Cref{prop:ExponentialEstimateHoriArc}, $D$ is determined by the $|\Phi|$-distance from a horizontal arc to zeros of $\Phi$. Now, for each $\Phi\in\SQD(\sigma)$,  $\widetilde{\gamma}_\Phi^h$ is at least $r/2$ away from zeros of $\Phi$ in the $|\Phi|$ metric, so the constant $D$ can be taken uniformly on $\SQD(\sigma)$. Moreover, we take the constant $C$ as follows.
	We set $r'>0$ so that, for every $\Phi\in\SQD(\sigma)$ and every $p\in Z(\Phi)$,
	\begin{equation*}
		B_\sigma(p,r')\subset B_{|\Phi|}(p,r/2).
	\end{equation*}
	By the same manner as the upper semicontinuity of $k_2$, we conclude that the map
	\begin{equation*}
		\Phi \mapsto \min_{p\in M_{r'}(\rho(\Phi))} |\nu(\Phi)|
	\end{equation*}
	is lower semicontinuous on $\QD(\sigma)$.
	Therefore, there exists a lower bound of
		$\displaystyle\min_{p\in M_{r'}(\rho(\Phi))} |\nu(\Phi)(p)|$
	for $\Phi\in\SQD(\sigma)$.

	Thus we find that $|\nu(\Phi)| \geq C_1 > 0$ and $\log |\nu(\Phi)|^{-1}\leq C_2 < +\infty$ on $\widetilde{\gamma}_\Phi^h\  (\subset M_{r'}(\rho(\Phi)))$ for every $\Phi\in\SQD(\sigma)$. Therefore we obtain the desired constant $C$, since it is given by $C_1$ and $C_2$ as in the proof of \Cref{prop:ExponentialEstimateHoriArc}.
\end{remark}

\subsection{The extended homeomorphism}

The following lemma is an extension of \cite[Lemma 4.7]{wolf1989teichmuller} for punctured surfaces.

\begin{prop}\label{lemma:ConvergenceEquivalence}
	Let $\{\rho_i\}\subset\Teich_{g,n}$ be a sequence diverging to $\infty$ (i.e. it leaves every compact in $\Teich_{g,n}$). Then $\{\pi\circ\ell(\rho_i)\}$ converges in $\PMF$ if and only if $\{\pi\circ\beta\Phi(\rho_i)\}$ converges in $\PMF$. Moreover, their limits in $\PMF$ coincide when they converges.
\end{prop}

\begin{proof}
	We can take finitely many essential simple closed curves $[\gamma_1],\ldots,[\gamma_k]\in\mathcal{C}$ and a positive number $\delta>0$ so that
	\begin{equation*}
		\sum_{j=1}^k i(\beta\Phi_0,[\gamma_j])>\delta>0
	\end{equation*}
	holds for every $\Phi_0\in\SQD(\sigma)$.

	Suppose that $\pi\circ\ell(\rho_n)$ converges. Then, there exists a sequence $\{\lambda_n\}\subset\R_{>0}$ such that $\lambda_n\ell(\rho_n)$ converges in $\R_{>0}^{\mathcal{C}}$. Therefore, there exists $B>0$ such that, for each $n\in\N$,
	\begin{align}
		B
			&>\sum_{j=1}^k \lambda_n\ell_{\rho_n}([\gamma_j])  \notag\\
			&>\lambda_n\sum_{j=1}^ki(\beta\Phi(\rho_n),[\gamma_j]) \ \ \ (\text{By \Cref{prop:FundamentalInequality}}) \label{eq:IN}\\
			&=\lambda_n\|\Phi(\rho_n)\|^{1/2}\sum_{j=1}^k i(\beta(\Phi(\rho_n)/\|\Phi(\rho_n)\|),[\gamma_j])  \notag\\
			&>(\lambda_n\|\Phi(\rho_n)\|^{1/2})\delta. \notag
	\end{align}
	Hence, we see that
	\begin{equation*}
		\lambda_n\cdot\eta([\gamma],\varepsilon,\|\Phi(\rho_n)\|)
			<(B/\delta)\cdot\eta\|\Phi(\rho_n)\|^{-1/2}.
	\end{equation*}
	By \Cref{prop:FundamentalInequality},
	\begin{equation*}
		\lim_{n\to\infty}\eta\|\Phi(\rho_n)\|^{-1/2} = c_0([\gamma],\varepsilon) <\varepsilon.
	\end{equation*}
 Morever, for any $\varepsilon>0$ and $[\gamma]\in\mathcal{C}$,
	\begin{align*}
		\lambda_ni(\beta\Phi(\rho_n),[\gamma])
			&=\lambda_n\|\Phi(\rho_n)\|^{1/2}i(\beta(\Phi(\rho_n)/\|\Phi(\rho_n)\|),[\gamma]) \\
			&<(B/\delta)\max_{\Phi_0\in\SQD(\sigma)}i(\beta\Phi_0,[\gamma]).
	\end{align*}
	Thus by \Cref{prop:FundamentalInequality}, for any $\varepsilon>0$ and $[\gamma]\in\mathcal{C}$,
	\begin{equation*}
		|\lambda_n\ell_{\rho_n}([\gamma])-\lambda_ni(\beta\Phi(\rho_n),[\gamma])|
			<(k_0-1)\lambda_n i(\beta\Phi(\rho_n),[\gamma])+\lambda_n\eta
	\end{equation*}
	and
	\begin{equation*}
		(k_0-1)\lambda_n i(\beta\Phi(\rho_n),[\gamma])\to 0
	\end{equation*}
	as $n\to\infty$. Therefore, for any $\varepsilon>0$ and $[\gamma]\in\mathcal{C}$,
	\begin{equation*}
		\lim_{n\to \infty}(\lambda_n\ell_{\rho_n}([\gamma])-\lambda_ni(\beta\Phi(\rho_n),[\gamma]))<\varepsilon.
	\end{equation*}
	Since we can take arbitraly small $\varepsilon>0$, we have
	\begin{equation*}
		\lim_{n\to \infty}\lambda_n\ell(\rho_n)=\lim_{n\to\infty}\lambda_n\beta\Phi(\rho_n)\ \  \text{in}\ \R_{\geq0}^\mathcal{C}.
	\end{equation*}
	Thus we conclude that
	\begin{equation*}
		\lim_{n\to\infty}\pi\circ\beta\Phi(\rho_n)=\lim_{n\to\infty}\pi\circ\ell(\rho_n).
	\end{equation*}
	We can show the converse in the same manner by starting with \cref{eq:IN}.
\end{proof}

Define a map $\psi:\overline{\Teich_{g,n}^{\mathrm{Th}}}\to\overline{\Teich_{g,n}^{\mathrm{harm}}}$ by
\begin{equation*}
\psi(x)=
\begin{cases}
	\left(\displaystyle\lim_{n\to\infty} \frac{\Phi(x_n)}{\|\Phi(x_n)\|},1\right)  &(x\in\partial_{\mathrm{Th}}\Teich_{g,n},\ \{x_n\}\subset \Teich_{g,n} \ \text{with}\ x_n\to x)\\[15pt]
	\left( \displaystyle\frac{\Phi(x)}{\|\Phi(x)\|}, \frac{4\|\Phi(x)\|}{1+4\|\Phi(x)\|} \right) &(x\in\Teich_{g,n}),
\end{cases}
\end{equation*}
where we use the polar coordinates in $\overline{\BQD(\sigma)}$ (i.e. for $(r,\theta)\in\overline{\BQD(\sigma)}$, $\theta\in\SQD(\sigma)$ and $r\in[0,1]$).

Our main theorem, \cref{thm:mainthm}, is shown from \Cref{lemma:ConvergenceEquivalence}. The proof is essentially same as the proof of \cite[Theorem 4.1]{wolf1989teichmuller}. However we write the proof here for the sake of completeness.

\begin{theorem}\label{thm:mainthm}
	The map $\psi$ is a homeomorphism.
\end{theorem}

\begin{proof}
	We first show that $\psi$ is well-defined on $\PMF$. For sequences $\{x_n\}$, $\{x_n'\}$ with $x_n,x_n'\to x \in \PMF$, there exist
	\begin{equation*}
		\lim_{n\to\infty} \pi\circ\beta\Phi(x_n)\  \text{and}\  \lim_{n\to\infty}\pi\circ\beta\Phi(x_n'),
	\end{equation*}
	and they coincide by \Cref{lemma:ConvergenceEquivalence}. Therefore there exist sequences $\{\lambda_n\}$, $\{\lambda_n'\}$ such that
	\begin{equation*}
		\lim_{n\to\infty}\lambda_n\beta\Phi(x_n)=\lim_{n\to\infty}\lambda_n'\beta\Phi(x_n').
	\end{equation*}
	Since $\beta$ is a homeomorphism,
	\begin{equation*}
		\lim_{n\to\infty}\lambda_n^2\Phi(x_n)=\lim_{n\to\infty}{\lambda'_n}^2\Phi(x_n').
	\end{equation*}
	Hence, we have
	\begin{equation*}
		\lim_{n\to\infty} \frac{\Phi(x_n)}{\|\Phi(x_n)\|}=\lim_{n\to\infty} \frac{\lambda_n^2\Phi(x_n)}{{\lambda_n}^2\|\Phi(x_n)\|}
		=\lim_{n\to\infty} \frac{{\lambda_n'}^2\Phi(x_n')}{{\lambda_n'}^2\|\Phi(x_n')\|}
		=\lim_{n\to\infty} \frac{\Phi(x_n')}{\|\Phi(x_n')\|}.
	\end{equation*}
	Thus, $\psi$ is well-defined.

	We secondly show that $\psi$ is continuous. In particular, we need to show $\psi$ is continuous at a point $x\in\PMF$. Let $\{x_n\}$ be a sequence in $\Teich_{g,n}$ with $x_n\to x$. Then, since $x_n$ diverges to $\infty$,
	\begin{equation*}
		\frac{4\|\Phi(x_n)\|}{1+4\|\Phi(x_n)\|} \to 1\ \ \text{as}\ \ n\to\infty.
	\end{equation*}
	Thus, we find that the second componet of $\psi$ is continuous. The first componet is continuous by the definition of $\psi$.

	We thirdly show that $\psi$ is injective. The injectivity on $\Teich_{g,n}$ follows from the injectivity of $\Phi$. Therefore, we need to show $\psi$ is injective on $\PMF$. Suppose that $\psi(x)=\psi(x')$ for $x,x'\in\PMF$. Let $x_n$ and $x_n'$ be sequences in $\Teich_{g,n}$ such that $x_n\to x\in\PMF$ and $x_n'\to x'\in\PMF$ as $n\to\infty$.
	Since $\psi(x)=\psi(x')$, we have
	\begin{equation*}
		\lim_{n\to\infty}\frac{\Phi(x_n)}{\|\Phi(x_n)\|}=\lim_{n\to\infty}\frac{\Phi(x_n')}{\|\Phi(x_n')\|}.
	\end{equation*}
	Therefore, we have
	\begin{equation*}
		\lim_{n\to\infty}\pi\circ\beta\Phi(x_n)=\lim_{n\to\infty}\pi\circ\beta\Phi(x_n').
	\end{equation*}
	By \cref{lemma:ConvergenceEquivalence}, we have
	\begin{equation*}
		\lim_{n\to\infty}\pi\circ\ell(x_n)=\lim_{n\to\infty}\pi\circ\ell(x_n').
	\end{equation*}
	This imples $x=x'$.
	Thus we find that $\psi$ is injective on $\PMF$.

	We forthly show that $\psi$ is surjective. The restriction $\psi|_{\Teich_{g,n}}$ is surjective, since $\Phi$ is a homeomorphism, so we need to show the surjectivity for $\SQD(\sigma)$.
	Taking each $\theta\in\SQD(\sigma)$, we set $x_n=\Phi^{-1}(n\theta)$. Then
	$\pi\circ \beta\Phi(x_n)=\pi\circ\beta(n\theta)=[F_{v}(\theta)] \in \PMF$ for every $n$. Therefore, we have
	\begin{equation*}
		\lim_{n\to\infty} \pi\circ\ell(x_n)=[F_v(\theta)]
	\end{equation*}
	by \Cref{lemma:ConvergenceEquivalence}. Thus,
	\begin{equation*}
		\psi([F_v(\theta)])=\left(\lim_{n\to\infty}\frac{\Phi(x_n)}{\|\Phi(x_n)\|},1\right)=(\theta,1).
	\end{equation*}
	Thus, $\psi$ is surjective.

	We finally show that $\psi^{-1}$ is continuous. Let $(\theta_n,r_n)\to(\theta,1)$ as $n\to\infty$ for each $\theta\in\SQD(\sigma)$, where $\theta_n\in \SQD(\sigma)$ and $r_n\in(0,1)$. Then
	\begin{equation*}
		\psi^{-1}(\theta_n,r_n)=\Phi^{-1}\left(\frac{r_n\theta_n}{4(1-r_n)}\right) \text{ and } \psi^{-1}(\theta,1)=[F_v(\theta)].
	\end{equation*}
	Therefore, setting $x_n=\psi^{-1}(\theta_n,r_n)$, we have
	\begin{equation*}
		\pi\circ\beta\Phi(x_n)=[F_v(\theta_n)].
	\end{equation*}
	Thus,
	\begin{equation*}
		\psi^{-1}(\theta,1)=[F_v(\theta)]=\lim_{n\to\infty}[F_v(\theta_n)]=\lim_{n\to\infty}\pi\circ\beta\Phi(x_n)=\lim_{n\to\infty}\psi^{-1}(\theta_n,r_n),
	\end{equation*}
	and this completes the proof.
\end{proof}

\begin{remark}
	This is another proof that the Thurston compactification $\overline{\Teich_{g,n}^{\mathrm{Th}}}$ is a closed ball of dimension $6g-6+2n$.
\end{remark}

\bibliography{HarmMapCptfy}
\bibliographystyle{alpha}

\end{document}